\documentclass[a4paper,12pt]{amsart}

\usepackage{amsmath,amsthm,amsfonts,amssymb,amscd,enumerate,ae,arydshln}
\usepackage[latin1]{inputenc}
\usepackage{graphicx}                
\usepackage[makeroom]{cancel}       
\usepackage{multirow}               
\usepackage[toc,page]{appendix}
\usepackage{color}

\newcommand{\F}{\mathbb{F}}

\renewcommand{\leq}{\leqslant}
\renewcommand{\geq}{\geqslant}

\newcommand{\omb}{\overline{\Omega}} 
\newcommand{\normal}{\trianglelefteq}

\DeclareMathOperator{\sym}{Sym}
\DeclareMathOperator{\aut}{Aut}
\DeclareMathOperator{\soc}{soc}

\DeclareMathOperator{\po}{P\Omega}
\DeclareMathOperator{\wre}{wr}

\DeclareMathOperator{\out}{Out}

\DeclareMathOperator{\supp}{supp}

\DeclareMathOperator{\hol}{Hol}
\newcommand{\comp}[2]{#1^{(#2)}}

\newcommand{\ves}{\vspace*{.4cm}}
\newcommand{\ir}{\underline{r}}

\newcommand{\il}{\underline{l}}

\newtheorem{coro}{Corollary}[section]
\newtheorem{lemma}[coro]{Lemma}

\newtheorem{teor}[coro]{Theorem}

\theoremstyle{definition}
\newtheorem{defi}[coro]{Definition}

\newtheorem{exem}[coro]{Example}

\newcommand{\cent}[2]{{\mathbb C}_{#1}(#2)}
\newcommand{\norm}[2]{{\mathbb N}_{#1}(#2)}

\newcommand{\interv}[1]{\underline{#1}}
\title[Transitive characteristically simple subgroups]{Transitive characteristically simple subgroups of finite quasiprimitive permutation groups}
\author{Pedro H. P. Daldegan}
\author{Csaba Schneider}

\address[C. Schneider]{Departamento de Matem\'atica\\
Instituto de Ci\^encias Exatas\\
Universidade Federal de Minas Gerais\\
Av.\ Ant\^onio Carlos 6627\\
Belo Horizonte, MG, Brazil\\
csaba@mat.ufmg.br\\
 www.mat.ufmg.br/$\sim$csaba}

\address[P. Daldegan]{Departamento de Matem\'atica\\
Centro Federal de Educação Tecnol\'ogica de Minas Gerais\\
Av. Amazonas 7675\\
Belo Horizonte, MG, Brazil\\
phpdaldegan@cefetmg.br\\
sig.cefetmg.br/sigaa/public/docente/portal.jsf?siape=1395846}

\begin{document}

\dedicatory{Dedicated to the memory
of Charles Sims}

\begin{abstract}
  The first main result of this paper is that
  a finite transitive nonabelian characteristically simple subgroup of a wreath product in
  product action must lie in the base group of the wreath product.
  This allows us to characterize nonabelian transitive characteristically
  simple subgroups $H$ of finite quasiprimitive permutation groups $G$.
  If the socle of $G$, denoted by $\soc(G)$, is nonabelian, then $H$ lies in $\soc(G)$.
  An explicit description is given for the possibilities of $H$ under the
  condition that $H$ does not contain a nontrivial normal subgroup of $\soc(G)$.
\end{abstract}



\thanks{The research presented in this paper formed a part of the first author's
  PhD project under the supervision by the second author. Daldegan was
  supported by a PhD scholarship awarded by the CAPES and  CNPq,
  while Schneider acknowledges
  the support of the CNPq grants {\em Universal} (project no.: 421624/2018-3)
  and {\em Produtividade em Pesquisa} (project no.: 308773/2016-0). We thank the referees for their careful reading of the manuscript.}

\maketitle

\section{Introduction}\label{sect:intro}

The characterization of transitive subgroups of finite permutation groups
is an important problem with several applications in algebraic combinatorics
and algebraic graph theory. Without the ambition of being exhaustive,
we recall some results in this research area.
Finite primitive permutation groups with a
transitive abelian subgroup were classified by Li~\cite{liabelian}, while finite
permutation groups with a transitive cyclic subgroup were described by
Li and Praeger~\cite{praegercyclic}.
Regular subgroups of finite primitive permutation groups
were classified by Liebeck, Praeger and Saxl~\cite{liebeckregular}
and by Baumeister~\cite{baumeisterregular}. A general description
of transitive subgroups of finite primitive
permutation groups was given by Liebeck, Praeger and
Saxl~\cite{liebecktransitive}.

In this paper we focus on transitive
nonabelian characteristically simple subgroups of finite quasiprimitive
permutation groups. 
Transitive simple subgroups of wreath products in
product action were classified by the second author in collaboration with
Baddeley and Praeger in~\cite{cartesiancsaba}.
The examples of~\cite{cartesiancsaba} consisted of two isolated groups and
of two infinite families.
Furthermore, it was proved by Baddeley and Praeger
in~\cite[Theorem~1.4]{badprgprimover}, that if $H$ is a nonabelian, nonsimple,
characteristically simple subgroup of a finite quasiprimitive almost simple
permutation group $G$, then either $H$ is intransitive or
$\soc(G)$ is an alternating group in its
natural action.
The
scarcity of such examples  gave the hope that
nonabelian characteristically simple subgroups of finite quasiprimitive groups
could be described.
This goal is achieved in this paper.

It is an easy consequence of the divisibility relations of
Lemma~\ref{prime} that if
$T$ is a transitive nonabelian simple subgroup of a wreath product in
product action, then $T$ is always contained in the base group of the wreath
product. Our first theorem generalizes this result to 
transitive characteristically simple groups.

\begin{teor}\label{imbase} 
  Let $\Gamma$ be a finite set such that $|\Gamma|\geq 2$, let $r\geq 2$, and let $W=\sym(\Gamma)\wre S_r$ be considered as
  a permutation group on $\Omega=\Gamma^r$ in product action. If $H$ is a transitive nonabelian characteristically
  simple subgroup of $W$, then $H$ is a subgroup of the base group; that is, $H\leq \sym(\Gamma)^r$.
\end{teor}

Since finite quasiprimitive permutation groups are often contained in
wreath products in product action, Theorem~\ref{imbase} leads to the following
more general result. 

\begin{teor}\label{main}
  Let $G$ be a finite quasiprimitive permutation group acting on
  $\Omega$ with nonabelian socle $S$ and let $H$ be a transitive nonabelian
  characteristically simple subgroup of $G$. Then $H\leq\soc(G)$.
\end{teor}

Once we know by Theorem~\ref{main} that $H$ lies in the socle of $G$,
a more detailed description of the possibilities of $H$ can be given
considering the possible O'Nan--Scott classes of $G$. For the description
of the O'Nan--Scott classes of quasiprimitive permutation groups see
Section~\ref{quasiprim}, while for the terminology
related to strips see Section~\ref{sect:2}. 
Given a natural number $r$, the symbol $\interv r$ denotes $\{1,\ldots,r\}$.

\begin{teor}\label{theorem3}
 Let $G$ be a finite quasiprimitive permutation group acting on $\Omega$ and 
 let $S=Q_1\times\cdots\times Q_r$ be the socle of $G$,
 where each $Q_i$ is isomorphic to a  simple group
 $Q$. Assume that $T$ is a nonabelian finite simple group, and let
 $H\cong T_1\times\cdots\times T_k\cong T^k$
 be a transitive subgroup of $G$ with $k\geq 2$. If
 $H$ does not contain a minimal normal subgroup of $S$, then,
 after possibly relabeling the $T_i$, one of the
 following holds.
 \begin{enumerate}
   \item[(0)] $G$ has type {\sc HA}, $T\cong \mbox{SL}(3,2)$, $|\Omega|=8^k$.
\item $G$ has type {\sc As}, $S= A_n$ and $G=A_n$ or
  $G=S_n$ acting naturally on $n$ points with $n\geq 10$.
\item $G$ has type {\sc Sd} or {\sc HS},
  $k=r=2$ and
  $T_i<Q_i$, for $i=1,2$. Moreover, the groups
  $Q$ and $T$ are described in one of the rows of Table~\ref{factiso}.
  Furthermore, if
  $$
  S_\alpha=\{(q,q\alpha)\mid q\in Q_1\}
  $$
  for some isomorphism $\alpha:Q_1\rightarrow Q_2$, then
  $Q_1=T_1(T_2\alpha^{-1})$.
\item $G$ has type {\sc Pa} and
  one of the following is valid.
  \begin{enumerate}
  \item $k=r$ and $T_i<Q_i$ for all $i\in\interv r$.
  \item The $T_i$ are pairwise disjoint strips in $S$ with
    $|\supp(T_i)|\in\{1,2\}$ for all $i$, with $|\supp(T_i)|=2$ for at least one $i$, and the groups $T$ and $Q$ are as in one of the rows of
    Table~\ref{table2}.
  \item $S\cong(A_n)^r$, where $|\Omega|=n^r$ and $n\geq 10$.
    \end{enumerate}
\item $G$ has type {\sc Cd} or HC, $k=r$ and
  $T_i<Q_i$ for all $i$; the groups
  $Q$ and $T$  are described in one of the rows of Table~\ref{factiso}.
  Further, a point stabilizer $S_\alpha$ is the direct product of pairwise disjoint strips
  $D$ with $|\supp(D)|=2$. If
  $$
  D=\{(q,q\alpha)\mid q\in Q_i\}
  $$
  is a strip in $S_\alpha$ with some isomorphism $\alpha:Q_i\rightarrow Q_j$,
  then $Q_i=T_i(T_j\alpha^{-1})$.
\end{enumerate}
\end{teor}

\begin{center}
\begin{table}
$$
\begin{array}{|c|c|c|}
\hline
& T & Q \\
\hline\hline
1& A_6 & A_6 \\
\hline
2 & M_{12} & M_{12} \\
\hline
3 & M_{12} & A_{12} \\
\hline
4 & \po_8^+(q) & \po_8^+(q) \\
\hline
5 & \po_8^+(q) & A_n\mbox{ where } n=|\po_8^+(q):\Omega_7(q)|\\
\hline
6 & \po_8^+(q) & \mbox{Sp}_8(2) \\\hline
7 & \mbox{Sp}_4(2^a),\ a\geq 2 & \mbox{Sp}_{4b}(2^{a/b}) \mbox{ where }\ b\mid a \\
\hline
8 & \mbox{Sp}_4(2^a),\ a\geq 2 & A_n \mbox{ where } n=|\mbox{Sp}_4(2^a):\mbox{Sp}_2(2^{2a})\cdot 2|\\
\hline
\end{array}
$$
\caption{Table for Theorem~\ref{theorem3}(3)(b)}\label{table2}
\end{table}
\end{center}

\begin{table}
  $$
\begin{array}{|c|c|c|}
\hline 
 & T & Q \\ \hline\hline
1& A_5 & A_6 \\ \hline
2& M_{11} & M_{12} \\ \hline
3& \Omega_7(q) & 
\po_8^+(q) \mbox{ with  }q\geq 2 \\ \hline
\end{array}
$$
\caption{Table for Theorem~\ref{theorem3}(2)-(4)}
\label{factiso}
\end{table}

In part~(1) of Theorem~\ref{theorem3}, we paraphrase
the aforementioned result~\cite[Theorem~1.4]{badprgprimover} by Baddeley and Praeger.

Note that we stated 
Theorem~\ref{theorem3} under the condition that $H$ contains no
minimal normal subgroup of $\soc(G)$,
which is the same condition as the one applied in~\cite{liebecktransitive}.
If we allow that $H$ can contain a minimal normal subgroup of $\soc(G)$, then the simple
components of $H$ and $\soc(G)$ must be isomorphic and $H$ is the direct
product of full strips of $\soc(G)$ (see Definition~\ref{stripdiag}). Such examples do arise and their
classification can be achieved by considering the factorization
$\soc(G)=H(\soc(G)_\alpha)$ where $\soc(G)_\alpha$ is a point stabilizer in
$\soc(G)$.
We also note that the group $G$ in Theorem~\ref{theorem3} is actually
primitive, except perhaps in part~(3). This shows that a finite
imprimitive quasiprimitive group $G$ very rarely contains a transitive
characteristically simple subgroup that does not contain a minimal normal
subgroup of $\soc(G)$.

In the proofs of Theorems~\ref{main} and~\ref{theorem3} we could
have relied more on the descriptions of transitive
subgroups of primitive permutations groups presented
in~\cite{liebecktransitive}. However,
since we are interested only in transitive characteristically simple 
subgroups, our arguments are simpler that those in~\cite{liebecktransitive},
and so we preferred a more direct approach.

Using Theorem~\ref{theorem3}, we obtain a characterization of
regular characteristically simple subgroups of quasiprimitive
permutation groups.

\begin{coro}\label{cor1}
  Suppose that $G$ and $H$ are as in Theorem~\ref{theorem3}.
If $H$ is regular then the O'Nan--Scott type of $G$ is either {\sc As} or
  {\sc Pa}.
\end{coro}

If we allow that $H$ may contain a minimal normal subgroup of $\soc(G)$, then we
obtain several examples of regular characteristically simple subgroups of
$G$ with $G$ having other O'Nan--Scott types. For example, if $G$ has
type {\sc Sd} and $H$ is a maximal normal subgroup of $\soc(G)$, then
$H$ is regular and characteristically simple.

Regular subgroups occur in the automorphisms groups of Cayley graphs.
Quasiprimitive subgroups of automorphisms of Cayley
graphs of finite simple groups were studied in~\cite{cayleyqp}.
Our results allow us to state the following corollary concerning
quasiprimitive subgroups of the automorphism groups of Cayley graphs of
characteristically simple groups.

\begin{coro}\label{cor3}
  Let $H$ be a finite nonabelian characteristically simple group and let
  $\Gamma$ be a noncomplete
  Cayley graph of $H$. Suppose that $G$ is a quasiprimitive
  subgroup of $\aut(\Gamma)$ with nonabelian socle such that $H\leq G$.
  Then either $H$ contains a minimal normal subgroup of $\soc(G)$ or the O'Nan--Scott type
  of $H$ is {\sc Pa}.
\end{coro}

Theorems~\ref{imbase} and~\ref{main}
are proved in 
Sections~\ref{sectionchar} and \ref{maintheorem}, respectively.
The proofs of Theorem~\ref{theorem3}, and Corollaries~\ref{cor1} and~\ref{cor3}
can be found in Section~\ref{sectioninclusions}.
Examples are described in Section~\ref{sect:examples} to
show that all possibilities described in Theorem~\ref{theorem3} arise.

\section{Subgroups of characteristically simple groups and factorizations}
\label{sect:2}

We say that a group $G$ is \textbf{characteristically simple} if it has no proper and nontrivial characteristic subgroups; that is, its only subgroups invariant under $\aut(G)$ are $1$ and $G$ itself. In particular, when $G$ is finite, this is equivalent to saying that $G$ is a direct product of isomorphic simple groups \cite[3.3.15]{robinson}.

If $A$ and $B$ are proper subgroups of a group $G$ such that $G=AB$, then we call this expression a
\textbf{factorization}
of $G$.
In a transitive permutation group $G$, a subgroup $H\leq G$ is transitive
if and only if $G=HG_\alpha$ where $G_\alpha$ is
the stabilizer of a point. Hence factorizations play a natural role
in the study of transitive subgroups of permutation groups and in this section
we collect some auxiliary results concerning subgroups and factorizations
of finite characteristically simple groups.

\begin{lemma}\label{qabt}
Let $Q$ be a finite simple group, and suppose that $Q=AB$ is a factorization of $Q$ 
 where $A\cong T^{s_1}$ and $B\cong T^{s_2}$ for some finite nonabelian simple group $T$ and positive integers $s_1$ and $s_2$. 
 Then $s_1=s_2=1$ and $Q$ and $T$ are as in Table \ref{factiso}.
\end{lemma}
\begin{proof}
  For a group $G$, let $p(G)$ denote the set of primes that divide
  $|G|$. We claim that $p(Q)=p(A)=p(B)$.  Clearly, $p(A)\subseteq p(Q)$ and
  $p(B)\subseteq p(Q)$. On the other hand,
  given a prime number $p\in p(Q)$, since 
  $|Q|=|A||B|/|A\cap B|$ and $|A|$ and $|B|$ are powers of $|T|$,
  the prime $p$ divides $|T|$, and so $p$ divides
  $|A|$ and $|B|$. 
  Thus $p(Q)=p(A)=p(B)$, as claimed. In the terminology
  of~\cite{factorizations}, the expression $Q=AB$ is a full
  factorization of the simple group $Q$ and such factorizations
  are described in~\cite[Theorem~1.1]{factorizations}.
  Keeping in mind that  $\mbox{Sp}_6(2)\cong\Omega_7(2)$, inspection
 of~\cite[Table I]{factorizations} shows that
  the only options
  where $|A|$ and $|B|$ are powers of the same finite simple group occur
  when $A\cong B\cong T$ and $Q$ and $T$ are as in Table \ref{factiso}.
\end{proof}

Next we introduce some terminology to describe diagonal subgroups
in wreath products.

\begin{defi}\label{stripdiag}
  Let $Q_1,\ldots,Q_r$ be groups and set
  $S=Q_1\times\cdots\times Q_r$. Consider, for $i\in\interv r$,
  the projections 
  $$\begin{array}{ccccc}
    \pi_i&\colon&S&\rightarrow&Q_i\\
    & &(q_1,\ldots,q_r)&\mapsto&q_i,
   \end{array}$$ 
and assume that $P\leq S$. 
\begin{enumerate}
 \item[(i)] $P$ is a \textbf{strip}
 \index{strip}%
 of $S$ if $P\neq 1$ and, for each $i\in\ir$, either the restriction of $\pi_i$ to $P$ is injective  or $P\pi_i=1$.
  We define the \textbf{support}
 of $P$ as $$\supp(P)=\{Q_i\mid P\pi_i\neq 1\}.$$ 
 \item[(ii)] A strip $P$ of $S$ is said to be \textbf{nontrivial}
 \index{strip!nontrivial}%
 if $|\supp(P)|>1$.
 \item[(iii)] A strip $P$ of $S$ is said to be a \textbf{full strip}
 \index{strip!full}%
 if $P\pi_i=Q_i$ for all $Q_i\in\supp(P)$.
 \item[(iv)] Two strips $P$ and $Q$ are
 said to be \textbf{disjoint} if $\supp(P)\cap\supp(Q)=\varnothing$.
 \item[(v)] $P$ is a \textbf{subdirect subgroup}
 \index{subgroup!subdirect}%
 of $S$ if $P\pi_i=Q_i$ for each $i\in\ir$.  
 \item[(vi)] $P$ is a \textbf{diagonal subgroup}
 \index{subgroup!diagonal}%
 of $S$ if the restriction of $\pi_i$ to $P$ is injective
  for each $i\in\ir$.
 \item[(vii)] $P$ is a \textbf{full diagonal subgroup}
 \index{subgroup!full diagonal}%
 of $S$ if $P$ is both a subdirect and a diagonal subgroup of $S$. 
\end{enumerate}
\end{defi}

Note that the definitions above depend on the given
direct decomposition of the group $S$. However, we will usually assume that
$S$ is the direct product of pairwise isomorphic nonabelian simple groups
and, unless explicitly stated otherwise,
the conditions in Definition~\ref{stripdiag} will be 
interpreted with respect to the unique finest direct decomposition of $S$.
The concepts introduced in Definition~\ref{stripdiag} are studied in more
depth in~\cite[Section~4.4]{bookcsaba}.
If $P$ is a strip of $S$, as above, and, for some $i\in\interv r$, the restriction $\pi_i|_P$ is injective, then  $P\pi_i\cong P$.


The first part of the following lemma appeared in Scott's
paper~\cite[Lemma~p.~328]{scott}, and is known as
Scott's Lemma (see also~\cite[Section~4.6]{bookcsaba} for several generalizations).
It describes the structure of the subdirect subgroups of a direct product of nonabelian simple groups. The second part 
can be found, for example, in \cite[Proposition 5.2.5(i)]{structureclassical}.
Note that Lemma~\ref{scottlemma} does not assume that the groups should be finite.

\begin{lemma}\label{scottlemma}
 Consider $S=Q_1\times\cdots\times Q_r$, where each $Q_i$ is a nonabelian simple group, and let 
 $P$ be a nontrivial subgroup of $S$. 
 \begin{enumerate}
 \item If $P$ is a subdirect subgroup of $S$, then $P$ is the direct product
   of full strips whose supports form a partition of $\{Q_1,\ldots,Q_r\}$.
  \item If $P$ is a normal subgroup of $S$, then $P=\prod_{j\in J} Q_j$, where $J\subseteq \ir$. 
 \end{enumerate}
\end{lemma}

The next result characterizes the factorizations $S=HD$ of finite nonabelian
characteristically 
simple groups $S$, in which $H$ is a nonabelian characteristically simple subgroup and $D$
is a full diagonal subgroup of $S$. 

\begin{lemma}\label{hab}
  Let $Q$ and $T$ be nonabelian finite simple groups,
  let $S=Q_1\times\cdots\times Q_r$ where $r\geq 2$ and
  $Q_i\cong Q$ for all $i$, 
and let $H\cong T^k$ be a nonabelian characteristically simple subgroup of $S$.
Consider, for $i\in\ir$, the projections $\pi_i\colon S\rightarrow Q_i$  and suppose that 
$1<H\pi_i<Q_i$ for all $i\in\ir$.
Assume, for $i=2,\ldots,r$, that $\alpha_i:Q_1\rightarrow Q_i$
is an isomorphism and define
$$
D=\{(q,q\alpha_2,\ldots,q\alpha_r)\mid q\in Q_1\}.
$$
If $DH=S$, then $r=k=2$ and $H=T_1\times T_2$ where
$T_1=H\cap Q_1$, $T_2= H\cap Q_2$ and $T_1\cong T_2\cong T$.
Further, in this case, 
$Q_1=T_1(T_2\alpha_2^{-1})$ and
the groups  $Q$ and $T$ are as in one of the rows of Table~\ref{factiso}.
 \end{lemma}
\begin{proof}
  By \cite[Lemma 8.16]{bookcsaba},
  $r\leq 3$. Further, if $r=3$, then the same lemma implies that
  $$
  Q_1=H\pi_1(H\pi_2\alpha_2^{-1}\cap H\pi_3\alpha_3^{-1})
  =H\pi_2\alpha_2^{-1}(H\pi_1\cap H\pi_3\alpha_3^{-1})=
  H\pi_3\alpha_3^{-1}(H\pi_1\cap H\pi_2\alpha_2^{-1}).
  $$
  Hence, in the terminology of~\cite{factorizations},
  the set $\{H\pi_1,H\pi_2\alpha_2^{-1},H\pi_3\alpha_3^{-1}\}$ is a
  strong multiple factorization of $Q_1$. 
 Noting that $H\pi_i\cong T^{s_i}$ with $s_i\geq 1$,
 and considering the classification of strong multiple factorizations
 of finite simple groups
 in \cite[Table V]{factorizations}, we obtain that $r=2$.
 Furthermore, $Q_1=(H\pi_1)(H\pi_2\alpha_2^{-1})$ where
 $H\pi_1\cong T^{s_1}$ and $H\pi_2\alpha_2^{-1}\cong H\pi_2\cong T^{s_2}$ with some
 $s_1,s_2\geq 1$.
 Now Lemma~\ref{qabt} implies that $H\pi_1\cong H\pi_2\cong T$
 and that the groups $Q$ and $T$ are as in Table~\ref{factiso}.
 Finally, note that $H\cap Q_i\unlhd H\pi_i$ holds for
 $i=1,2$. Hence either $H\cap Q_i=H\pi_i$ or
 $H\cap Q_i=1$. Thus, if $H\neq (H\cap Q_1)\times
 (H\cap Q_2)$, then $H\cong H\pi_1\cong
 H\pi_2\cong T$, and hence the factorization $S=DH$ is impossible,
 since
 $$
 |D||H|=|Q||T|<|Q|^2=|S|.
 $$
 Thus $H= (H\cap Q_1)\times
 (H\cap Q_2)$ must hold. Setting $T_1=H\cap Q_1$ and $T_2=H\cap Q_2$, the
 rest of the lemma follows.
 \end{proof}

The following result was originally proved in~\cite[Lemma~2.2]{badprgprimover};
see also~\cite[Theorem~1.2]{uniform} and~\cite[Section~4.8]{bookcsaba}
for generalizations.

\begin{lemma}\label{uniform2}
  Let $T$ be a nonabelian finite simple group and let $X$ and $Y$ be
  direct products of pairwise disjoint  nontrivial full strips in $T^r$
  with $r\geq 2$. Then $XY\neq T^r$.
\end{lemma}

The following technical lemma, which characterizes
certain factorizations related to the Mathieu group $M_{12}$,
will be used in the last section.

\begin{lemma}\label{lemmam12}
 Let $S=Q^4$ and $A_1,A_2,A_3,A_4\leq Q$ such that $Q=M_{12}$ and each $A_i\cong M_{11}$.
 Let $X=\{(p,p,q,q)\mid p,q\in Q\}$ and $Y=\{(a_1,a_2,a_2\psi,a_4)\mid a_i\in A_i\}$ be subgroups 
 of $S$, where $\psi\colon A_2\rightarrow A_3$ is an isomorphism. Then $S\neq XY$.
\end{lemma}
\begin{proof}
 Let $D_1=\{(p,p,1,1)\mid p\in Q\}$ and $D_2=\{(1,1,q,q)\mid q\in Q\}$. 
 We have that $X=D_1\times D_2\cong Q^2$ is the direct product of two full strips of $S$ and $Y\cong (M_{11})^3$ is the direct product of
 three strips of $S$, where the second strip is a diagonal subgroup of $Q^2$. 
 
 Assume that $Q^4=XY$ and consider the projections $\pi_{12}\colon S\rightarrow Q^2$ and $\pi_{34}\colon S\rightarrow Q^2$, where $\pi_{12}$ projects 
onto the first two coordinates and $\pi_{34}$ projects onto the last two coordinates.
Applying these projections to $Q^4=XY$, we obtain that 
\begin{align}
 Q^2=&D_1(A_1\times A_2),\label{q21}\\
 Q^2=&D_2(A_3\times A_4).\label{q22}
\end{align}
It follows from Lemma~\ref{hab} that $Q=A_1 A_2=A_3 A_4$.
Thus,
$$
|A_1\cap A_2|=|A_3\cap A_4|=|M_{11}|^2/|M_{12}|=660.
$$
Set $C_1=A_1\cap A_2$ and $C_2=(A_3\cap A_4)\psi^{-1}$. 
Note that 
$$X\cap Y=\{(c,c,c\psi,c\psi)\mid  c\in C_1\cap C_2\}\cong C_1\cap C_2.$$
Since $Q^4=XY$, 
$$|X\cap Y|=\frac{|X||Y|}{|Q|^4}=\frac{|M_{11}|^3}{|M_{12}|^2}=55.$$
As $X\cap Y\cong C_1\cap C_2\leq A_2$ and
$$|C_1 C_2|=\frac{|C_1||C_2|}{|X\cap Y|}=\frac{660^2}{55}=7920=|A_2|,$$
we find that $A_2=C_1 C_2$. Since $M_{11}$ has a unique  conjugacy
class of subgroups of order  660 \cite[p.18]{atlas},
$C_1$ and $C_2$ are conjugate in $A_2$, which
contradicts to $A_2=C_1C_2$. Then $Q^4\neq XY$ and the result is proved.
\end{proof}

\section{Quasiprimitive permutation groups}\label{quasiprim}

A permutation group is \textbf{quasiprimitive} \index{quasiprimitive} if all its nontrivial normal subgroups are transitive.
For example, primitive permutation groups are quasiprimitive. Since a transitive 
simple group is always quasiprimitive, but not necessarily primitive, the class of quasiprimitive permutation groups
is strictly larger than the class of primitive permutation groups. 


Finite primitive and quasiprimitive groups were classified by
the respective versions of the O'Nan--Scott Theorem;
see~\cite[Chapter~7]{bookcsaba}.
In this classification, we distinguish between 8 classes of finite primitive groups,
namely {\sc HA}, {\sc HS}, {\sc HC}, {\sc SD}, {\sc CD}, {\sc PA}, 
{\sc AS}, {\sc TW}, and 8~classes of finite quasiprimitive groups, namely 
{\sc HA}, {\sc HS}, {\sc HC}, {\sc Sd}, {\sc Cd}, {\sc Pa}, 
{\sc As}, {\sc Tw}. As the notation suggests, quasiprimitive groups in classes {\sc HA}, {\sc HS} and {\sc HC} are always primitive. Here we only give a brief summary of these classes;
the interested reader can find a more detailed treatment
in~\cite[Chapter~7]{bookcsaba}.
The type of a primitive or quasiprimitive group $G$ can be
recognized from the structure and the 
permutation action of its socle, denoted $\soc(G)$. Let $G\leq\sym(\Omega)$ 
be a quasiprimitive 
permutation group, let $M$ be a minimal normal subgroup of $G$, and 
let $\omega\in\Omega$. Note that $M$ is a characteristically simple group.
If $M$ is nonabelian, then by a subdirect subgroup of $M$ we mean one that
 is 
subdirect with respect to the unique finest
direct decomposition of $M$ (see Definition~\ref{stripdiag}).
For a group $G$, the holomorph $\hol G$ is the semidirect product
$G\rtimes\aut G$ viewed as a permutation group acting on the set $G$;
see~\cite[Section~3.3]{bookcsaba} for more details. The
main characteristics of $G$ and $M$ in each primitive and quasiprimitive type
are as follows.

{\sc HA:} $M$ is abelian, $\cent GM=M$ and $G\leq \hol M$.
A quasiprimitive permutation group in this class is always primitive.

{\sc HS:} $M$ is nonabelian, simple, and regular; 
$\soc(G)=M\times\cent GM\cong M\times M$ and 
$G\leq\hol M$. Such a quasiprimitive permutation group is always primitive.

{\sc HC:} $M$ is nonabelian, nonsimple, and regular; $\soc(G)=M\times\cent
GM\cong M\times M$ and $G\leq\hol M$. A quasiprimitive permutation group
of this type is always primitive.

{\sc Sd:} $M$ is nonabelian and nonsimple; $M_\omega$ is a 
simple subdirect subgroup of $M$ and $\cent GM=1$. If, 
in addition, $G$ is primitive, then the type of $G$ is {\sc SD}.

{\sc Cd:} $M$ is nonabelian and nonsimple; $M_\omega$ is a 
nonsimple subdirect subgroup of $M$ and  $\cent GM=1$. If, 
in addition, $G$ is primitive, then the type of $G$ is {\sc CD}.

{\sc Pa:} $M$ is nonabelian and nonsimple; $M_\omega$ is  
not a subdirect subgroup of $M$ and $M_\omega\neq 1$; $\cent GM=1$. If, 
in addition, $G$ is primitive, then the type of $G$ is {\sc PA}.

{\sc As:} $M$ is nonabelian and simple; $\cent GM=1$. If, 
in addition, $G$ is primitive, then the type of $G$ is {\sc AS}.

{\sc Tw:} $M$ is nonabelian and nonsimple; $M_\omega=1$; $\cent GM=1$. If,
in addition, $G$ is primitive, then the type of $G$ is {\sc TW}.

Note that if $G$ is a primitive permutation group of type HS or HC, then
there exists a primitive group $\overline G$ of type SD or CD, respectively,
that contains $G$ as a subgroup of index~2
(see \cite[Corollary~3.11]{bookcsaba}). Hence in the proof of the
inclusion  $H\leq G$ in Theorem~\ref{main}, the cases when
$G$ has type HS or HC can be reduced to the cases when $G$ has type SD or CD,
respectively.

\section{Examples of transitive characteristically simple subgroups}
\label{sect:examples}

In this section we describe some
examples to show that all the possibilities described
in Theorem~\ref{theorem3} arise. In Examples~\ref{ex0}--\ref{ex4},
$T$ is a nonabelian finite simple group.

\begin{exem}[$G$ has type HA]\label{ex0}
  The following example goes back to at least D.~G.~Higman~\cite[Lemma~4]{higman}.
  Suppose that $R=\mbox{SL}(3,2)$. Then $R$ acts irreducibly on
  $V=\F_2^3$ and set $G=V\rtimes R$. It is well-known (see for
  example~\cite[Proposition~5.2]{prginclusionprim}) that $G$ has a transitive
  simple subgroup $T$ isomorphic to $R$. Letting $W=G\wre S_r$ with $r\geq 2$,
  the direct product $T^r$ is a transitive characteristically simple
  subgroup in the
  primitive permutation group $W$ of type HA acting on $\Omega=\F_2^{3r}$. Hence $T^r$
  is contained in any primitive subgroup of $\F_2^{3r}\rtimes \mbox{GL}(2,3r)$
  that contains $W$.
  These inclusions are as in Theorem~\ref{theorem3}(0).
  \end{exem}

\begin{exem}[$G$ has type {\sc As}]\label{ex1}
  Set $H=T^k$ with $k\geq 2$.
  Suppose that $X$ is a corefree subgroup of $H$ and consider the right coset
  space $\Omega=[H:X]$. Then $H$ can be viewed as a transitive
  subgroup of $Q=\mbox{Alt}(\Omega)$
  which is a  primitive permutation group of type AS.
  These examples satisfy the conditions of Theorem~\ref{theorem3}(1).
\end{exem}

\begin{exem}[$G$ has type {\sc Sd} or HS]\label{ex2}
  Suppose that $T$ and $Q$ are as in one of the rows of Table~\ref{factiso}.
  Set $D$ to be the diagonal subgroup
  $$
  D=\{(q,q)\mid q\in Q\}
  $$
  of $Q^2$ and let $\Omega$ denote the right coset space $[Q^2:D]$.
  The normalizer $\norm{\sym(\Omega)}{Q^2}$ is a primitive group of
  type SD whose abstract group structure is $(Q^2\cdot\mbox{Out}(Q))
  \rtimes C_2$. Suppose that
  $G\leq\norm{\sym(\Omega)}{Q^2}$ 
   such that
  $G$ contains $Q^2$. Then $G$ is a primitive permutation group of type
  SD or HS depending on the projection of $G$ on $C_2$. 
    Assume that $T_1$ and $T_2$ are subgroups of $Q$ such that
    $T_1\cong T_2\cong T$ and $T_1T_2=Q$ and $H=T_1\times T_2$. Viewing
    $H$ as a subgroup of $Q^2$, easy calculation shows that
  $DH=Q^2$ and so $H$ is a transitive
    characteristically simple subgroup of $G$.
    These examples are as in Theorem~\ref{theorem3}(2).
\end{exem}

\begin{exem}[$G$ has type {\sc Pa}]\label{ex3}
  Examples for Theorem~\ref{theorem3}(3) can be constructed by wreathing
  smaller examples. Suppose that $T$ is a simple transitive
  subgroup of a nonabelian simple group
  $Q$ acting on a set $\Gamma$. Then, for $r\geq 2$, the wreath product
  $W=Q\wre S_r$ contains the transitive characteristically simple subgroup $T^r$.
  Furthermore, the inclusion $T^r\leq W$ is as in Theorem~\ref{theorem3}(3)(a).

  Suppose that $H$ and $Q$ are as in Example~\ref{ex1}. Then, for $r\geq 2$,
  $H^r$ is contained in $Q\wre S_r$ and this inclusion is as in
  Theorem~\ref{theorem3}(3)(c).

  Finally, let $Q$ and $T$ be as in one of the rows of Table~\ref{table2}.
  Then $T$ can be embedded as a transitive subgroup of $G=Q\wre S_2$, and so,
  for $r\geq 2$, $T^r$ is a transitive characteristically simple subgroup of
  $W=G\wre S_k$. The inclusion $T^r\leq W$ is as in
  Theorem~\ref{theorem3}(3)(b).
\end{exem}

\begin{exem}[$G$ has type {\sc Cd} or HC]\label{ex4}
  Suppose that $G$ and $H$ are as in Example~\ref{ex2}. If $r\geq 2$, then
  the wreath product $W=G\wre S_r$ is a primitive group of type CD or
  HC  (depending on the type of $G$) that contains
  the transitive characteristically simple subgroup $H^r=T^{2r}$.
  The inclusion $H\leq W$ is as in Theorem~\ref{theorem3}(4).
  \end{exem}

\section{The proof of Theorem~\ref{imbase}}\label{sectionchar}

This section contains the proof of Theorem \ref{imbase}.
We start by  stating a number theoretic result
which is a corollary of Legendre's Formula for the largest prime-power that
divides $n!$.

\begin{lemma}\label{prime}
 Given natural numbers $p,n\geq 2$, the following are valid.  
 \begin{enumerate}
  \item $p^n\nmid n!$.
  \item If $p^{n-1}\mid n!$, then $p=2$ and $n$ is a power of $2$.
  \item $4^{n-1}\nmid n!$.
 \end{enumerate}
\end{lemma}

Next we review some concepts related to subgroups of wreath products in product
action.
Let $\Gamma$ be a finite set such that $|\Gamma|\geq 2$, let $r\geq 2$, and let $W=\sym(\Gamma)\wre S_r$ be considered as
a permutation group on $\Omega=\Gamma^r$ in product action.
An element of $W$ is written as $(a_1,\ldots,a_r)b$ where
$a_i\in\sym(\Gamma)$ for all $i\in\interv r$, and $b\in S_r$.
Suppose that $X$ is a subgroup of $W$. For $j\in\ir$, 
we define the \textbf{$j$-th component}
\index{$j$-th component}%
$\comp Xj$  of $X$ as follows. Suppose that
$W_j$ is the stabilizer in $W$ of $j$ under the permutation representation 
$\pi\colon W\rightarrow S_r$. Then 
\begin{equation}\label{dp}
W_j=\sym(\Gamma)\times(\sym(\Gamma)\wre S_{r-1}),
\end{equation}
where the first
factor of the direct product acts on the $j$-th coordinate, while the second
factor acts on the other coordinates. In particular, `$S_{r-1}$' is taken 
to be the stabilizer of $j$ in $S_r$. We define $\comp Xj$ as the 
projection of $X_j=X\cap W_j$ onto the first factor of $W_j$.
We view $\comp Xj$ as a subgroup of $\sym(\Gamma)$.
The following theorem was stated in~\cite[Theorem~1.2]{embedding};
see also~\cite[Corollary~5.17]{bookcsaba}.

\begin{teor}\label{thtrans}
If $X$ is a transitive subgroup
of $W$, then each component of $X$ is transitive on $\Gamma$. 
Moreover, if $X$ acts transitively on
$\ir$, 
then each component of the intersection $X \cap (\sym(\Gamma))^r$ 
is transitive on $\Gamma$.
\end{teor}

We turn to the proof of Theorem~\ref{imbase}.

\begin{proof}[Proof of Theorem~\ref{imbase}]
Suppose that $H=T_1\times\cdots\times T_k=T^k$ for some nonabelian 
finite simple group $T$. Suppose, as above, 
that $\pi:W\rightarrow S_r$ is the natural 
projection. Let $B$ be the base group $\sym(\Gamma)^r$ of $W$. Then 
$B=\ker\pi$. Assume for contradiction that $H\not\leq B$; that is $H\pi\neq 1$. 

First we assume that $H\pi$ is transitive on $\ir$. The case when 
$H\pi$ is intransitive will be treated afterwards. 
Set $H_B=H\cap B$. Then $H_B$ is a normal subgroup of $H$ 
and, by Lemma \ref{scottlemma}(2), it is of the form $T^s$, with some $s$. Further, $H=H_B\times \overline{H}_B$ 
where similarly $\overline{H}_B=T^{k-s}$,  and $\overline{H}_B$ acts
transitively and faithfully by $\pi$ on $\ir$. For $j\in\ir$, consider the component $\comp {H_B}j$ as a permutation group on 
$\Gamma$. By Theorem~\ref{thtrans}, $\comp {H_B}j$ is transitive on 
$\Gamma$ for all $j$.

\medskip

\noindent{\em Claim 1.} $\comp{H_B} j\cong H_B$ for all $j$.
\medskip

\noindent{\em Proof of Claim 1.}
  Suppose, for $j\in\ir$, that $\sigma_j:\sym(\Gamma)^r\rightarrow\sym(\Gamma)$
  denotes the $j$-th coordinate projection. Then  $\comp{H_B}j\cong H_B/(\ker\sigma_j\cap H_B)$. 
Let $m$ be an element of $\ker\sigma_1\cap H_B$. Thus $m=(1,m_2,\ldots,m_r)$ with
$m_j\in\comp{H_B}j$. Let $j\in\ir$. 
Since $\overline{H}_B$ is transitive on $\ir$, there is 
some element $g=(g_1,\ldots,g_r)h$ of $\overline{H}_B$ such that $1(g\pi)=1h=j$. 
Then 
$
m^g=(1,m_2,\ldots, m_r)^g=
(1,m_2^{g_2},\ldots,m_r^{g_r})^h,
$
and so the $j$-th coordinate of $m^g$ is 1. Hence $m^g\in\ker\sigma_j\cap H_B$, and then
$(\ker\sigma_1\cap H_B)^g\leq\ker\sigma_j\cap H_B.$
Analogously, the same argument above shows that 
$\ker\sigma_j\cap H_B\leq(\ker\sigma_1\cap H_B)^g$, 
and so $\ker\sigma_j\cap H_B=(\ker\sigma_1\cap H_B)^g$.
On the other hand, $\ker\sigma_1\cap H_B$ is a subgroup
of $H_B$ and $\overline H_B$ centralizes $H_B$, and so 
$\ker\sigma_1\cap H_B=\ker\sigma_j\cap H_B$  for all $j$. Thus $\ker\sigma_1\cap H_B$ acts trivially on 
$\Omega$, and so $\ker\sigma_1\cap H_B=1$, which gives $\ker\sigma_j\cap H_B=1$ for all $j$.
Therefore, $\comp{H_B} j\cong H_B$ for all~$j$, which proves Claim 1.
\medskip

Thus the restrictions to $H_B$ of the projection maps $\sigma_j$ are 
monomorphisms. Then $\beta_j=\sigma_1^{-1}\sigma_j:\comp{H_B}1\rightarrow 
\comp{H_B}j$ is an isomorphism for all $j$. As a consequence, every element $m\in H_B$
can be expressed uniquely as $m=(y,y\beta_2,\ldots,y\beta_r)$, for some $y\in\comp{H_B}1$.

\medskip

\noindent{\em Claim 2.} For all $j\in\ir$, there is some element
$x_j\in\sym(\Gamma)$ such that $y\beta_j=y^{x_j}$ for all $y\in\comp{H_B}1$.
\medskip

\noindent{\em Proof of Claim 2.}
Suppose that $y\in \comp{H_B}1$. Then $m=(y,y\beta_2,\ldots,y\beta_r)\in 
H_B$. Let $j\in\ir$ and, using the transitivity of $\overline{H}_B$ on $\ir$,  
suppose that $g=(g_1,\ldots,g_r)h\in \overline{H}_B$ is such that $1(g\pi)=1h=j$. 
Then $g$ centralizes $m$ and hence 
$$
(y,y\beta_2,\ldots,y\beta_r)=m^g=(y^{g_1},(y\beta_2)^{g_2},\ldots,(y\beta_r)^{g_r})^h.
$$
Comparing the $j$-th coordinates in the two sides of the last 
equation, we find that $y\beta_j=y^{g_1}$. Taking $x_j=g_1$, thus we have $y\beta_j=y^{x_j}$, which proves Claim 2.
\medskip

\noindent{\em Claim 3.} If $\Sigma$ is a $H_B$-orbit in $\Omega$, then $|\Sigma|=|\Gamma|$.
\medskip

\noindent{\em Proof of Claim 3.}
Since $H$ is transitive on $\Omega$ and $H_B\unlhd H$, all the 
$H_B$-orbits have the same size. Hence it suffices to show the claim 
for just one $H_B$-orbit. Choose the elements $1,x_2,\ldots,x_r$ as in the 
previous claim, let $\gamma\in\Gamma$ and consider the element
$\omega=(\gamma,\gamma x_2,\ldots,\gamma x_r)$. 
Suppose that $m\in H_B$. By the previous claim, $m$ has
the form $m=(y,y^{x_2},\ldots,y^{x_r})$ for some $y\in\comp{H_B}1$. Hence
$\omega^m=(\gamma y,\gamma y x_2,\ldots,\gamma y x_r)$. Thus $m$ stabilizes 
$\omega$ if, and only if, $y\in\sym(\Gamma)$ stabilizes $\gamma$. Thus 
$(H_B)_\omega=(\comp{H_B}1)_\gamma$. So by applying the Orbit-Stabilizer Theorem twice, 
and using that $|H_B|=|\comp{H_B}1|$ and that $\comp{H_B}1$ is transitive on $\Gamma$, we have
$$|\omega^{H_B}|=\frac{|H_B|}{|(H_B)_\omega|}=\frac{|\comp{H_B}1|}{|(\comp{H_B}1)_\gamma|}=|\Gamma|,$$
and so Claim 3 is proved.

\medskip

\noindent{\em Claim 4.} The case when $H\pi$ is transitive is impossible.
\medskip

\noindent{\em Proof of Claim 4.}
$H_B$ is a normal subgroup of $H$ and every $H_B$-orbit has size $|\Gamma|$. 
Hence the number of $H_B$-orbits is $|\Gamma|^{r-1}.$ 
Since $H$ is transitive
on $\Omega$, $\overline{H}_B$ is transitive on the set of $H_B$-orbits and 
hence $|\Gamma|^{r-1}\mid |\overline{H}_B|$. Since $\overline{H}_B$ has a
faithful action on $\ir$, 
this leads to $|\Gamma|^{r-1}\mid r !$. Now, since $\Gamma$ is an orbit
for the characteristically simple group $\comp{H_B}1$, we find that 
$|\Gamma|\geq 5$. Hence $|\Gamma|$ is divisible by $p$, where 
$p$ is either an odd prime or $p=4$, which is a contradiction by Lemma \ref{prime}. Then Claim 4 is proved.
\medskip

This completes the proof for the case when $H\pi$ is a transitive subgroup of 
$S_r$. Let us now turn to the case when $H\pi$ is intransitive. 
Recall that $B$ is the base group of $W$. Assuming that $H\not\leq B$
implies that there exists an $H\pi$-orbit $\Delta$ in $\ir$ with size at 
least 2. Set $\overline\Delta=\ir\setminus\Delta$ and
$r_1=|\Delta|$. Then
$H$ can be embedded into
the direct product 
$$W_1\times W_2=
\left(\sym(\Gamma)\wre S_{r_1}\right)\times \left(\sym(\Gamma)\wre S_{r-r_1}
\right)$$
such that the projection $H_1$ of $H$ into $W_1$ acts transitively
on $\underline{r_1}$. Now, since $H$ is transitive on $\Gamma^r$, $H_1$ is also transitive 
on $\Gamma^{r_1}$. Further, as $H$ is characteristically simple, so is $H_1$.
Hence using the theorem in the case when $H\pi$ is transitive gives 
a contradiction. Therefore, $H\leq B$. 
\end{proof}

Since in several classes of quasiprimitive permutation groups, the individual
groups are subgroups in wreath products in product action, Theorem~\ref{imbase}
leads to the following corollary.

\begin{coro}\label{embed}
  Let $G\leq\sym(\Omega)$
  be a finite quasiprimitive permutation group with nonabelian
  socle $S$ and let $H$ be a transitive nonabelian
  characteristically simple subgroup of $G$.
  Let $\alpha\in\Omega$ and assume that $S=Q_1\times\cdots\times Q_r$, where $r\geq 2$ and the $Q_i$ are pairwise $G$-conjugate normal subgroups of $S$ such that 
 \begin{equation}\label{salphaembed}
 S_\alpha=(Q_1\cap S_\alpha)\times\cdots\times(Q_r\cap S_\alpha).
 \end{equation}
 Then $H\leq \norm G{Q_i}$ for all $i\in\ir$. Further, if the $Q_i$ are simple groups, then $H\leq S$.
\end{coro}
\begin{proof}
Set $\Gamma$ to be the right coset space $[Q_1\colon Q_1\cap S_\alpha]$.
By~\cite[Theorem~4.24]{bookcsaba}, we may assume without loss of generality
that $G$ is a subgroup of $W=\sym(\Gamma)\wre S_r$ acting in product
action on $\Gamma^r$, and so $H$ can also be viewed
as a transitive subgroup of $W$. By Theorem~\ref{imbase},
$H$ is a subgroup of the base group $\sym(\Gamma)^r$ of $W$.
Now~\cite[Theorem~4.24]{bookcsaba} also implies that the conjugation
action of $G$ on the set $\{Q_1,\ldots,Q_r\}$ is equivalent to its action
on $\interv r$ induced by the natural projection $\pi:W\rightarrow S_r$.
Since $\ker\pi=\sym(\Gamma)^r$, we have that $H$ acts trivially on
$\{Q_1,\ldots,Q_r\}$ by conjugation, and so $H\leq\norm G{Q_i}$ holds for all
$i$.

Suppose now that the $Q_i$ are simple. Since $\cent G{\soc(G)}=1$
and since $H\leq\norm G{Q_i}$ for all $i$, 
the group $H$ can be viewed as a subgroup of
$$
\aut S\cap \left(\bigcap_{i=1}^r\norm{\aut S}{Q_i}\right)=\prod_i(\aut Q_i).  
$$
Therefore
$$
H/(H\cap S)\cong
(HS)/S\leq \prod_i (\aut Q_i)/Q_i.
$$
Since Schreier's Conjecture  holds, the group on
the right-hand side of the last display is  soluble.
On the other hand, if $H\cap S\neq H$, then
$H/(H\cap S)$ is nonabelian characteristically simple,
which is impossible. Hence $H\cap S= H$, which
is equivalent to $H\leq S$.
\end{proof}

\section{The proof of Theorem \ref{main}}\label{maintheorem}

In this section we prove Theorem~\ref{main}.
Suppose that $G\leq\sym(\Omega)$ is a finite quasiprimitive
permutation group with nonabelian socle $S$ and let
$H$ be a transitive nonabelian characteristically simple subgroup
of $G$. We are required to show that $H\leq S$.
Our strategy is to analyze each of the possible O'Nan--Scott classes of $G$.

\ves
\textbf{$\boldsymbol G$ has type A{\scriptsize S}:}
$S$ is a simple group and $S\leq G\leq\aut(S)$. 
Then
$$H/(H\cap S)\cong (HS)/S\leq \aut(S)/S=\out(S).
$$
As $S$ is simple, it follows from
Schreier's Conjecture that $\out(S)$ is soluble. Since $H$ is nonabelian and characteristically simple,
$H\cap S=H$. Therefore $H\leq S=\soc(G)$, 
as desired.

\ves
\textbf{$\boldsymbol{G}$ has type HS:} $S=Q_1\times Q_2$
where $Q_1$ and $Q_2$ are simple minimal normal subgroups of $G$.
Furthermore, $S\leq G\leq\hol(Q_1)$
and
$$H/(H\cap S)\cong
(H S)/S\leq G/S\leq \hol(Q_1)/S\cong\out(Q_1).$$
Therefore, arguing as we did for the type A{\scriptsize S}, it follows from 
Schreier's Conjecture that $H\cap S=H$, and so
$H\leq S$.

\ves

\textbf{$\boldsymbol{G}$ has type T{\scriptsize W}:}
$S$ is regular, and so $S$ clearly satisfies equation~\eqref{salphaembed} with
respect to its finest direct decomposition into the direct product of
simple groups.  
Thus, by Corollary \ref{embed},  $H\leq\soc(G)$.

\ves

For the next four O'Nan--Scott types, namely for HC,
P{\scriptsize A}, S{\scriptsize D}, and
{\sc Cd}, we will assume that
$S=Q_1\times\cdots\times Q_r$ where $r\geq 2$ and
the $Q_i$ are pairwise isomorphic
nonabelian simple groups. We let $Q$ denote the common isomorphism type of
the $Q_i$. We define  $\pi_i$ as the coordinate projection
$\pi_i\colon S\rightarrow Q_i$.

\ves

\textbf{$\boldsymbol{G}$ has type HC:}
$S=S_1\times S_2$ where $S_1$ and $S_2$ are nonsimple, regular
minimal normal subgroups of $G$.  Then $r$ is even
and, after possibly reordering the $Q_i$, a point stabilizer in
$S$ is a direct product $D_1\times\cdots\times D_{r/2}$ where
each $D_i$ is a full diagonal subgroup in $Q_{i}\times
Q_{r/2+i}$. Since $r\geq 4$,
Corollary~\ref{embed} applies to the direct decomposition
$S=(Q_1\times Q_{r/2+1})\times\cdots\times (Q_{r/2}\times Q_{r})$
and this gives that $H$ normalizes $Q_i\times Q_{r/2+i}$ for all $i\in\interv{r/2}$.
Since $\cent G{S}=1$, the conjugation action of $H$ on $S$ embeds
$H$ into
$$
X=\aut(Q_1\times Q_{r/2+1})\times\cdots\times\aut(Q_{r/2}\times Q_{r}).
$$
Now $X/S$ is a subgroup of $((\out Q)\wre C_2)^r$, which is a
soluble group by Schreier's Conjecture. Now the usual argument implies that
$H\leq\soc(G)$.

\ves

\textbf{$\boldsymbol{G}$ has type P{\scriptsize A}:}
for all $\alpha\in\Omega$,
$S_\alpha$ is not a subdirect subgroup
of $S$ and $S$ is not regular.
In general, the groups of this class do not satisfy
equation~\eqref{salphaembed} in Corollary \ref{embed}, but it is well-known,
for $\alpha\in\Omega$,
that $G$ has a faithful quotient action on the right coset space
$\omb=[S:P]$ where $P=(S_\alpha\pi_1)\times\cdots\times(S_\alpha\pi_r)$.
Since $H$ also acts transitively on $\omb$ and $P$ is a point stabilizer in
$S$ under this action, equation~\eqref{salphaembed} holds for $P$ with the
finest direct product decomposition of $S$, and so $H\leq S$ must hold.

\ves
\textbf{$\boldsymbol{G}$ has type S{\scriptsize D}:} this is the most difficult
O'Nan--Scott type to deal with. In this case, 
for $\alpha\in\Omega$, $S_\alpha$ is a simple subdirect subgroup
of $S$. Since $S$ is transitive on $\Omega$ and $S_\alpha\cong Q$,  $|\Omega|=|Q|^{r-1}$.
 In this case, $G$ can be considered  as a subgroup  of 
 $\overline{G}=(S\cdot\out(Q))\rtimes S_r$, where $S_r$ permutes
 by conjugation
 the factors of $S$ and $\out(Q)$ acts on $S\cong Q^r$ diagonally;
 see~\cite[Section~7.4]{bookcsaba}.
 
 Consider the extension $\overline{S}=S\cdot\out(Q)$. We have that $\overline{G}$ permutes the elements in $\Sigma=\{Q_1,\ldots,Q_r\}$ and the kernel of this action
 is precisely $\overline{S}$. If we denote by $H_0$ the kernel of $H$ acting on $\Sigma$, we obtain that 
 $H_0=H\cap\overline{S}$. Since $H$ is characteristically simple, we have by Lemma \ref{scottlemma}(2) that $H_0\cong T^{k_0}$ for some integer $k_0$, and there exists a normal subgroup $H_1$ of $H$ such that 
 $H=H_0\times H_1$. It follows from the Isomorphism Theorem and
 from the definition of $\overline{S}$ that 
 \[H_0/(H_0\cap S)\cong (H_0 S)/S\leq\overline{S}/{S}\cong\out(Q).\]
Since $\out(Q)$ is soluble by Schreier's Conjecture, and $H_0$ is nonabelian and characteristically simple, we conclude that $H_0=H_0\cap S$, which means that $H_0\leq S$.   
 
If $H_1=1$, then, since $H_0\leq S$, we obtain at once that $H\leq S$.
Hence it suffices to prove that $H_1=1$. Suppose that 
 $H_1\neq 1$. Since $H_1\cap\overline{S}=1$, $H_1$ permutes the elements in $\Sigma$ faithfully, and so $|H_1|\mid r!$.
 In particular, since the size of the smallest nonabelian simple group is 60, we have $r\geq 5$.
 
 We claim that $H_0\neq 1$. In fact, if that is not the case, then $H=H_1$, so $H_1$ is transitive on $\Omega$.
  Then applying the Orbit--Stabilizer Theorem and the transitivity of $S$, we obtain
  \[\frac{|H_1|}{|(H_1)_\alpha|}=|\Omega|=|Q|^{r-1},\]
  so $|Q|^{r-1}\mid |H_1|$. Since $|H_1|\mid r!$, we obtain that $|Q|^{r-1}\mid r!$. Since the order of every finite nonabelian simple group is divisible 
  by four, this implies that $4^{r-1}\mid r!$, which is not possible by Lemma \ref{prime}. 
  Therefore, $H_0\neq 1$.
  
 Let us analyze the action of $H_0$ on $\Omega$. Since $H_0\normal H$ and $H$ is transitive, the orbits of $H_0$  
 form a block system for $H$. In particular, the  $H_0$-orbits have the 
 same size. By the Orbit--Stabilizer Theorem, it follows that
 $|(H_0)_\alpha|$ is independent of $\alpha$. 
 Therefore, the number of $H_0$-orbits on $\Omega$ is equal to
 \begin{equation}
   \label{orbeq}
   \frac{|Q|^{r-1}}{|\alpha^{H_0}|}=\frac{|Q|^{r-1}}{|T|^{k_0}}|(H_0)_\alpha|.
 \end{equation}
  Since $H=H_0\times H_1$ and $H$ is transitive on $\Omega$, we have that $H_1$ is transitive on the set of $H_0$-orbits.
 Then the Orbit--Stabilizer Theorem gives that the number of $H_0$-orbits divides $|H_1|$. 
 Therefore, the number of $H_0$-orbits divides $r!$.
  From the Isomorphism Theorem we have that
 \begin{equation}\label{pisi}
  Q_i\geq H_0\pi_i\cong H_0/(\ker\pi_i\cap H_0)\cong T^{s_i},
 \end{equation}
where $s_i\geq 0$ for all $i\in\ir$. 

We claim that $s_i\leq 1$ for all $i$.
Suppose, on the contrary, that there exists $i\in\ir$ such that $s_i\geq 2$.
In particular, $Q$ contains a subgroup isomorphic to $T^2$.
As every finite simple group has a cyclic Sylow subgroup
(see~\cite[Theorem~4.9]{simple}),
we can choose a prime $p$ such that the Sylow $p$-subgroups of $Q$ are cyclic.
Since a Sylow $p$-subgroup of $T^2$ is contained
in a Sylow $p$-subgroup of $Q$, we have that $p\nmid |T|$.
 In particular, $p\neq 2$.
Considering the number of $H_0$-orbits in~\eqref{orbeq},
we obtain that $p^{r-1}$ divides the number of
$H_0$-orbits, and so $p^{r-1}$ divides $r!$.
However, this contradicts Lemma \ref{prime}. Therefore,
$s_i\leq 1$ for all $i$, as claimed.

Since $H_0\neq 1$, $s_i= 1$ for some $i\in\ir$. Thus \eqref{pisi} gives that $Q$ has a subgroup isomorphic to $T$. 
Since each $T_i\leq H_0$ is simple,
each $T_i\leq H_0$ is a strip of $S$. We assert that if $i\neq j$, then $\supp(T_i)\cap\supp(T_j)=\varnothing$. 
In fact, if there is $m\in\ir$ such that $T_i\pi_m\cong T$ and $T_j\pi_m\cong T$, then $(T_i\times T_j)\pi_m\cong T^2$.
However, this is impossible, since  $s_m\leq 1$. 
Therefore
\begin{equation}\label{h0di}
 H_0=T_1\times\cdots\times T_{k_0},
\end{equation}
 where each $T_i\cong T$ is a diagonal subgroup of $$\prod_{Q_j\,\in\,\supp(T_i)}Q_j,$$ in such a way that $\supp(T_i)\cap\supp(T_j)=\varnothing$ for all $i\neq j$.
  As a consequence, we obtain that $k_0\leq r$. 
  Assume first that $k_0<r$; we will treat the case $k_0=r$ separately.
 Since $|T|\mid|Q|$, we have by \eqref{orbeq} that $|Q|^{r-k_0-1}$ divides the number of $H_0$-orbits in $\Omega$. As $H_1$ is transitive on the
 set of $H_0$-orbits, this implies that $|Q|^{r-k_0-1}$ divides $|H_1|$.
 
 Let $d_i=|\supp(T_i)|$.  
 Moreover, let $m_1$ be the number of factors $T_i$ for which $d_i\geq 5$, and let $m_2$ be the number of factors $T_i$ such that 
 $d_i<5$. So $m_1+m_2=k_0$. Relabeling if necessary, we can write
 $$H_0=T_1\times\cdots\times T_{m_1}\times T_{m_1+1}\times\cdots\times T_{m_1+m_2},$$
 such that $d_i\geq 5$ if and only if $i\leq m_1$. Set $m_3=r-\sum^{k_0}_{i=1}d_i$; that is,
 $m_3$ is the number of factors $Q_i$ such that $H_0\pi_i=1$. 
 
 Since $H_1$ centralizes $H_0$, we have that $H_1$ centralizes each $T_i$.
 for each $i\leq k_0$. As, for $h_1\in H_1$ and $i\in\interv {k_0}$, we have
 $$(\supp(T_i))^{h_1}=\supp({T_i}^{h_1})=\supp(T_i),$$
 we conclude that each $\supp(T_i)$ is $H_1$-invariant. In particular,
 $H_1$ acts by conjugation on $\supp (T_i)$ for all $i\in\interv{k_0}$, and, since
 $H_1$ is a nonabelian characteristically simple group, this action
 is trivial whenever $|\supp(T_i)|< 5$.
 Hence $H_1$ acts trivially on 
 $$\supp(T_{m_1+1})\dot{\cup}\cdots\dot{\cup}\supp(T_{k_0}).$$ 
 Since the action of $H_1$ on $\Sigma$ is faithful, we obtain that  
 \begin{equation}\label{h1}
 |H_1|\mid (d_1!)\cdots(d_{m_1}!)(m_3!). 
 \end{equation}
As we assumed $H_1\neq 1$, we have from~\eqref{h1} that either $m_1=0$ and $m_3\geq 5$, or $m_1\neq 0$. In both cases 
 we conclude that 
 \begin{equation}\label{sumeven}
  d_1+\cdots+d_{m_1}+m_3-m_1-1>0.
 \end{equation}
  Recall that $|Q|^{r-k_0-1}$ divides $|H_1|$. So~\eqref{h1} implies that
 \begin{equation}\label{orbm1}
  |Q|^{r-k_0-1}\mid(d_1!)\cdots(d_{m_1}!)(m_3!).
 \end{equation}
  On the other hand, we have that $r\geq d_1+d_2+\cdots+d_{m_1}+m_2+m_3$. So 
 $$|Q|^{d_1+d_2+\cdots+d_{m_1}+m_2+m_3}\mid |Q|^r.$$ 
 Since~\eqref{sumeven} is valid and $k_0=m_1+m_2$, we obtain
 $$|Q|^{(d_1-1)+\cdots+(d_{m_1}-1)+m_3-1}=|Q|^{d_1+d_2+\cdots+d_{m_1}+m_2+m_3-m_1-m_2-1}\mid |Q|^{r-k_0-1}.$$
 Therefore, using equation \eqref{orbm1}, we obtain that
 $$|Q|^{(d_1-1)+\cdots+(d_{m_1}-1)+m_3-1}\mid (d_1!)\cdots(d_{m_1}!)(m_3!).$$
 Since $4\mid|Q|$, the previous line gives that
 $$2^{d_1}\cdots 2^{d_{m_1}}2^{m_3}\mid (d_1!)\cdots(d_{m_1}!)(m_3!),$$
 which is a contradiction by Lemma \ref{prime}. This implies that $H_1=1$, which means that 
 if $k_0<r$, then $H\leq S$.
 
 Now consider the case where $k_0=r$. Then \eqref{h0di} implies that $d_i=1$ for all $i\in\ir$. 
 Since each $\supp(T_i)$ is $H_1$-invariant, we have that $H_1$ acts faithfully and trivially on $\Sigma$, 
 thus $H_1=1$. Therefore if $k_0=r$, we also obtain $H=H_0\leq S$.
 
 Therefore, if $G$ has type {\sc Sd}, then $H\leq S$.

\ves

\ves
\textbf{$\boldsymbol{G}$ has type C{\scriptsize D}:} for
$\alpha\in\Omega$, 
$S_\alpha$ is a subdirect subgroup of $S$, but it is not simple.
So Lemma \ref{scottlemma} 
implies that there exist  sets 
 $\overline{\Sigma}=\{S_1,\ldots,S_q\}$ and $\{D_1,\ldots,D_q\}$, where $q\geq 2$, each $D_i$ is a full diagonal subgroup of $S_i$
 and $S_i=\prod_{Q_j\in\supp(D_i)}Q_j$, such that
 $S=S_1\times\cdots\times S_q$, $G_\alpha$ acts transitively by conjugation on $\overline{\Sigma}$ and, 
 considering the projections $\overline{\pi}_i\colon S\rightarrow S_i$, we have that
 \begin{equation}
  S_\alpha=S_\alpha\overline{\pi}_1\times\cdots\times S_\alpha\overline{\pi}_q,
 \end{equation}
 where each $S_\alpha\overline{\pi}_i=D_i$.
 By~\cite[Theorem~11.13]{bookcsaba}, $G$ is permutationally isomorphic
 to a subgroup of a wreath product of the form
$W=G_0\wre S_q$ acting on $\Gamma^q$ in product action
where $G_0$ is a quasiprimitive permutation group on $\Gamma$ of type
{\sc Sd}. Further, setting $S=S_1$, $\soc(G_0)=S$ and
$\soc(G)=\soc(W)=S^q=S_1\times\cdots\times S_q$.
By Corollary~\ref{embed}, $H\leq\norm{G}{S_i}$ for all $i$, and
hence $H$ lies in the base group $(G_0)^q$ of $W$.

For $i\in\interv q$, let $\sigma_i:(G_0)^q\rightarrow G_0$ denote the $i$-th projection map. Then $H\sigma_i$ is a transitive characteristically simple
subgroup of $G_0$ acting on $\Gamma$. Since $G_0$ is quasiprimitive of
{\sc Sd} type, we find that $H\sigma_i\leq \soc(G_0)=S$. Since this is true
for all $i$, we obtain that $H\leq S^q=\soc(G)$.

\ves
This concludes the proof of Theorem \ref{main}.
We note that the cases when $G$ has type HS or HC could have been reduced
to the types {\sc Sd} and {\sc Cd}, respectively, as explained at
the end of Section~\ref{quasiprim}. However, we chose not to do this, since
these arguments are significantly easier than the one given for the type
{\sc Sd}.

\section{The proofs of Theorem~\ref{theorem3}
and Corollaries~\ref{cor1} and~\ref{cor3}}\label{sectioninclusions}

In this section we prove Theorem~\ref{theorem3} and Corollaries~\ref{cor1} and~\ref{cor3}.

\begin{proof}[The proof of Theorem~\ref{theorem3}]
(0) Suppose that $G\leq\sym(\Omega)$ has type HA. Then $G$ has a unique minimal normal
subgroup $V$ which is elementary abelian of order $p^d$.
Suppose that $H=T^k$ is a nonabelian transitive
characteristically simple subgroup of $G$. Then $H\cap V$ is a normal subgroup
of $H$ which is elementary abelian, and so, by Lemma~\ref{scottlemma}(2),
$H\cap V=1$. Hence one may consider
the group $X=V\rtimes H$. Since both $V$ and $H$ are transitive on
$\Omega$, so is $X$.

Though the group $X$ is transitive on $\Omega$, it may be imprimitive.
Suppose that $\Delta$ is a maximal proper $X$-block
in $\Omega$ (if $X$ is primitive, then $\Delta$ is a singleton) and let $\mathcal B$ denote the corresponding $X$-invariant
block system.
For a subgroup $Y$ of
$\sym(\Omega)$ preserving $\mathcal B$, let $Y^{\mathcal B}$ denote the permutation group
of $\sym(\mathcal B)$ induced by $Y$. By the maximality of $\Delta$,
$X^{\mathcal B}$ is a primitive group in which $V^{\mathcal B}$ is a transitive
abelian normal
subgroup and $H^{\mathcal B}$ is a transitive nonabelian characteristically
simple subgroup. In particular, $V^{\mathcal B}$ and $H^{\mathcal B}$ are nontrivial
and $V^{\mathcal B}$ is regular. Since $X=VH$, we obtain that
$X^{\mathcal B}=V^{\mathcal B}H^{\mathcal B}$. Furthermore, $V^{\mathcal B}\cap H^{\mathcal B}$ is an elementary
abelian normal subgroup in the nonabelian characteristically simple group
$H^{\mathcal B}$, which gives that $X^{\mathcal B}=V^{\mathcal B}\rtimes H^{\mathcal B}$. As $X^{\mathcal B}$
is primitive of HA type, the conjugation action of $H^{\mathcal B}$ induces
an irreducible linear group on $V^{\mathcal B}$. If $|V^{\mathcal B}|=p^{d_0}$, then
the group $H^{\mathcal B}$ has a transitive permutation representation of degree
$p^{d_0}$ and also a faithful irreducible representation of degree $d_0$ over the field
of $p$ elements. Now~\cite[Theorem~1.1]{barbaraha} implies that
$H^{\mathcal B}\cong T\cong\mbox{SL}(3,2)$, $p=2$, $d_0=3$ and the stabilizer $K$
in $H^{\mathcal B}=T$ is a subgroup of index~eight\footnote{The cited theorem of Baumeister contains a misprint
  and $\mbox{SL}(3,3)$ is written instead of $\mbox{SL}(3,2)$.}. Since, up to conjugacy,
$K$ is the unique subgroup of $T=\mbox{SL}(3,2)$ with $2$-power index,
we find that the stabilizer of a point $\alpha\in\Omega$ in $H$ must
be of the form $K^k$, and so $|\Omega|=8^k$.

In the rest of this proof, let $G$ be a finite quasiprimitive permutation group on $\Omega$ of type HS, HC, A{\scriptsize S}, T{\scriptsize W},
P{\scriptsize A}, S{\scriptsize D} or C{\scriptsize D}, and let $H\cong
T^k$
be a transitive nonabelian characteristically simple subgroup of $G$
where $k\geq 2$ and $T$ is a nonabelian finite simple group.
Assume that $S=Q_1\times\cdots\times Q_r$
is the socle of $G$, where each $Q_i\cong Q$ for a nonabelian simple group $Q$. 
Consider the projections $\pi_i\colon\soc(G)\rightarrow Q_i$ of $\soc(G)$ onto its direct factors.
According to Theorem \ref{main}, $H\leq S$ and we assume that
$H$ does not contain a nontrivial normal subgroup of $\soc(G)$.
If $G$ had type {\sc Tw}, then $S$ would be regular, and no
proper subgroup of $S$
would be transitive. Therefore, under these conditions,
the type of $G$ cannot be {\sc Tw}. The rest of the proof of
Theorem~\ref{theorem3} is by considering each possible O'Nan--Scott type
for $G$.

(1) In this case, $S\cong Q$ is a nonabelian simple group.
For $\alpha\in\Omega$, the factorization $S=HS_\alpha$ holds.
Now~\cite[Theorem 1.4]{badprgprimover} implies that $S=A_n$ with $n\geq 10$
acting naturally on $n$ points and hence $G=A_n$ or $G=S_n$ must follow.

(2) $G$ has type {\sc Sd} or HS.
Since $S$ is transitive on $\Omega$ and $S_\alpha\cong Q$,  $|\Omega|=|Q|^{r-1}$.
For a fixed $j\in\ir$, denote 
$$\overline{Q}_j=Q_1\times\cdots\times Q_{j-1}\times Q_{j+1}\times\cdots\times Q_r.$$
Given ${i_0}\in\ir$, \textit{a priori}, we have three options: $H\pi_{i_0}=1$, $H\pi_{i_0}=Q_{i_0}$ or $1<H\pi_{i_0}<Q_{i_0}$.

We claim, for all $i_0\in\interv r$, that $1<H\pi_{i_0}<Q_{i_0}$.
First we assume that  $H\pi_{i_0}=1$ for some ${i_0}\in\ir$.
  Without loss of generality, assume that ${i_0}=1$.
  Then $H$ is a transitive subgroup of $\overline Q_1$, which is a
  regular subgroup. Since no proper subgroup of $\overline Q_1$ is transitive,
  $H$ would have to be equal to $\overline Q_1$, which is impossible, as
  we assume that $H$ contains no nontrivial normal subgroup of $S$. Hence
  $H\pi_{i_0}\neq 1$ holds for all $i_0\in\interv r$.

Consider now the case in which $H\pi_{i_0}=Q_{i_0}$ for some ${i_0}\in\ir$. 
In this case $Q$ is a composition factor of $H$, which means by the Jordan-H\"older Theorem that $Q\cong T$, and so $k\leq r$.
Further, the transitivity of $H$ implies that  $k=r-1$ or $k=r$.
If $k=r$, then $H=S$, which is impossible in our conditions.
Assume now that $k=r-1$. By the previous paragraph, $H\pi_{i_0}\neq 1$
for all $i_0$, and so $H\pi_{i_0}=Q_{i_0}$ must hold for all $i_0$. Therefore
$H$ is a subdirect subgroup of $S$. By Lemma~\ref{scottlemma}, $H$ must be
the direct product of full strips with pairwise disjoint supports. Since
$k=r-1$, either $r=2$ and $k=1$ or one of these strips must be equal to a direct factor of $S$. The latter possibility cannot hold by our conditions.
In the former case, $Q_1\times Q_2=S_\alpha H$ is a factorization with
two full diagonal subgroups, which is not possible, by Lemma~\ref{uniform2}.

Therefore $1<H\pi_j<Q_j$ must hold for all $j\in\ir$ as claimed. 
We have that $S_\alpha$ is a full diagonal subgroup of $S$. Since $H$ is transitive, $S=HS_\alpha$. Now Lemma~\ref{hab} implies that $k=r=2$, 
$H=T^2=T_1\times T_2$ where $T_1< Q_1$ and $T_2< Q_2$
and $Q$ and $T$ are described in one of the rows of Table~\ref{factiso}. 
Assuming, as we may, that
$$
S_\alpha=\{(q,q\alpha)\mid q\in Q_1\}
$$
for some isomorphism $\alpha:Q_1\rightarrow Q_2$, the same lemma implies
that $Q_1=T_1(T_2\alpha^{-1})$.

(3) $G$ has type {\sc Pa}.
As $S_\alpha$ is nontrivial and is not a subdirect subgroup of $S$,
$1<S_\alpha\pi_i<Q_i$ holds for all $i$ and the projections $S_\alpha\pi_i$
are permuted by $G_\alpha$ transitively. Furthermore, since $H\pi_i$ is a
homomorphic image of $H$, we have that $H\pi_i=T^{s_i}$ with some $s_i\geq 0$.
Since $H$ is transitive, we have that $S_\alpha H=S$, which implies that
$s_i\geq 1$ for all $i\in\interv r$. Further, each minimal normal subgroup
$T_i$ of $H$ is a strip in $S$. Thus we may consider the support $\supp(T_i)$.
  The transitivity of $H$ implies that
  \begin{equation}\label{facq}
    Q_i=(HS_\alpha)\pi_i=(H\pi_i)(S_\alpha\pi_i)=T^{s_i}(S_\alpha\pi_i).
  \end{equation}

\ves
\textbf{Case 1:} Suppose that $s_{i_0}\geq 2$ for some $i_0\in\ir$. 
In this case, the factorization in~\eqref{facq}
and~\cite[Theorem 1.4]{badprgprimover} implies
$Q\cong A_n$ and $S_\alpha\pi_{i_0}\cong A_{n-1}$ where $n\geq 10$.
Since the $S_\alpha\pi_i$ are pairwise isomorphic,
$S_\alpha\pi_{i}\cong A_{n-1}$ holds for all $i\in\ir$.
Set $P=S_\alpha \pi_1\times\cdots\times S_\alpha \pi_r$.
In particular, $P\cong (A_{n-1})^r$.

We claim that $S_\alpha=P$.
Assume that $S_\alpha\neq P$. Since $P$ is a nonabelian characteristically simple group and $S_\alpha$ is a subdirect subgroup of $P$,
by Scott's Lemma (Lemma \ref{scottlemma}), $S_\alpha$ is the direct product of diagonal subgroups
$S_\alpha=D_1\times\cdots\times D_l$
for some $l\leq r$. Renumbering if necessary and using that $S_\alpha\neq P$, we can assume that $\supp(D_1)=\{Q_1,\ldots,Q_m\}$ with
$2\leq m\leq r$. Consider the projection $\overline{\pi}\colon S\rightarrow Q_1\times Q_2$.
Since $S=S_\alpha H$, we have
$Q_1\times Q_2=S\overline{\pi}=(S_\alpha\overline{\pi})(H\overline{\pi})$,
where $S_\alpha\overline{\pi}=\{(x,x\alpha)\mid x\in S_\alpha\pi_1\}$ for
an isomorphism $\alpha:S_\alpha\pi_1\rightarrow S_\alpha\pi_2$.
Since $n\geq 10$, the automorphisms of $A_n$ are induced by
conjugations by elements of $S_n$~\cite[2.4.1]{wilson}.
Thus we can extend the isomorphism $\alpha$ to an isomorphism
$\overline\alpha$ between $Q_1$ and $Q_2$. Now
$$Q_1\times Q_2=\overline{D}(H\overline{\pi})=\overline{D}(H\pi_1\times H\pi_2),$$
where $\overline{D}=\{(x,x\overline{\alpha})\mid x\in Q_1\}$
is a full diagonal subgroup in $Q_1\times Q_2$.
By Lemma \ref{hab}, the possibilities for $Q$ and $T$ are in
Table~\ref{factiso}.
Since  $Q\cong A_n$ with $n\geq 10$, we obtain a contradiction.
Therefore $S_\alpha=P$, as desired.

Applying the Orbit-Stabilizer Theorem, we see that
$$|\Omega|=\frac{|S|}{|S_\alpha|}=n^r.$$
So if $s_i\geq 2$ for some $i\in\ir$, then $S\cong (A_n)^r$, $S_\alpha\cong(A_{n-1})^r$ and
$|\Omega|=n^r$, where $n\geq 10$. Hence in this case we obtain Theorem~\ref{theorem3}(3)(c).

\ves
\textbf{Case 2:} Suppose that $s_i=1$ for all $i\in\ir$.
Since $T_i$ is simple, we have that $T_i$ is a strip of $S$ for all $i\in\ir$.
Moreover, the supports of the $T_i$ are pairwise disjoint. In fact, if for some $l$
we have $T_i\pi_l\cong T\cong T_j\pi_l$ for distinct $i,j\in\ir$, then
$T^2\cong(T_i\times T_j)\pi_l\leq H\pi_l\cong T$,
which is  absurd. Then the supports $\supp(T_i)$ are pairwise disjoint and we can write 
$$\begin{array}{lcl}
 T_1 & \leq & Q_1\times\cdots\times Q_{l_1},\\
 T_2 & \leq & Q_{l_1+1}\times\cdots\times Q_{l_1+l_2},\\
 \vdots & \vdots & \vdots \\
 T_k & \leq & Q_{l_1+l_2+\cdots+l_{k-1}+1}\times\cdots\times Q_{l_1+l_2+\cdots+l_k}.\\
\end{array}$$

First suppose that $l_i\geq 2$ for some $i\in\ir$. Renumbering, if necessary, assume that $l_1\geq 2$. 
Write $l=l_1$ and
consider the projection map $\overline{\pi}\colon S\rightarrow Q_1\times\cdots\times Q_l$.
As $S=HS_\alpha$, it follows that $(H\overline{\pi})(S_\alpha\overline{\pi})=Q_1\times\cdots\times Q_l$.
Write $L=S_\alpha\pi_1\times\cdots\times S_\alpha\pi_l$. 
Since $H\overline{\pi}=T_1$ and $S_\alpha\overline{\pi}\leq L$,
 $T_1L=Q_1\times\cdots\times Q_l$.
Therefore, $T_1$ is a transitive subgroup of
$Q_1\times\cdots\times Q_l$ under its faithful action
by right multiplication on the right
coset space $[Q_1\times\cdots\times Q_l\colon L]$.
On the other, hand $Q_1\times\cdots\times Q_l$ can be embedded into
the quasiprimitive permutation group $W=Q_1\wre S_l$ acting in product
action on $[Q_1:S_\alpha\pi_1]^l$. Hence $T_1$ is a
transitive simple subgroup of a wreath product in product action. 
According to \cite[Theorem 1.1]{cartesiancsaba}
and~\cite[Theorem~1.1(b)]{innatelywreath}, $l=2$ and $T$ and
$Q$ are as in one of the rows of Table~\ref{table2}.
Thus, in this case, Theorem~\ref{theorem3}(3)(b) is valid.

Now suppose that $l_i=1$ for all $i\in\ir$. Then $k=r$ and $T_i<Q_i$ for all $i\in\ir$.
In this case, Theorem~\ref{theorem3}(3)(a) holds.

(4)
$G$ has type {\sc Cd} or HC.
Lemma \ref{scottlemma} 
gives two sets 
$$
\overline{\Sigma}=\{S_1,\ldots,S_l\}\quad\mbox{and}
\quad \overline{D}=\{D_1,\ldots,D_l\}
$$
 where $l\geq 2$, each $D_i$ is a full diagonal subgroup of $S_i$
 and $S_i=\prod_{Q_j\in\supp(D_i)}Q_j$, such that
 $S=S_1\times\cdots\times S_l$ and 
 $S_\alpha=D_1\times\cdots\times D_l\cong Q^l$.

 We claim that $1<H\pi_j<Q_j$ for all $j\in\ir$.
Suppose that this is not the case and assume first that
$H\pi_{i_0}=Q_{i_0}$ for some $i_0\in\ir$.
In this case $Q$ is a composition factor of $H$, which means by the Jordan-Hölder 
Theorem  that $Q\cong T$, and so $k\leq r$.
Therefore $H$ is the direct product of disjoint full strips in $S$.
Since $H$ is transitive, $S=S_\alpha H$. Now Lemma~\ref{uniform2} implies 
that $H$ must contain a direct factor of $S$. In our case, this is
impossible.

Suppose now that $H\pi_{i_0}=1$ for some $i_0\in\ir$.
Let, for $i\in\interv l$,
$\bar\pi_i$ denote
the coordinate projection $\bar\pi_i:S\rightarrow S_i$.
We may assume that $Q_{i_0}\in\supp(D_1)$. Applying $\bar\pi_1$
to the factorization $S=S_\alpha H$, we obtain that
$S_1=D_1(H\bar\pi_1)$ and $H\bar\pi_1$ is contained in $X=\prod_{Q_i\in\supp(D_1)\setminus\{Q_{i_0}\}}
Q_i$. Now $S_1=D_1X$ is a factorization with $D_1\cap X=1$, which gives
that $H\bar\pi_1=X$. Therefore $H\pi_j=Q_j$ must hold for some
$j\neq i_0$. Hence by the analysis in the previous paragraph,
this is also impossible.
 Therefore $1<H\pi_j<Q_j$ holds for all $j\in\ir$, as claimed.

Considering  the
factorization
$$
S_i=D_i(H\bar\pi_i)
$$
we obtain, as in part~(2),
that $|\supp(D_i)|=2$ for all $i\in\il$. That is, $r$ is even and $S_i\cong Q^2$,
$H\pi_i\cong T$
where and $Q$ and $T$ are as in one of the rows of Table~\ref{factiso}.

Renumbering, if necessary, assume that 
$$S=\underbrace{Q_1\times Q_2}_{S_1}\times\underbrace{Q_3\times Q_4}_{S_2}\times\cdots\times\underbrace{Q_{r-1}\times Q_r}_{S_{r/2}}.$$

We claim that $H=H\pi_1\times\cdots\times H\pi_r$.
It suffices to prove that $H\pi_i\leq H$ for all $i\in\ir$. So assume the opposite,
that is, $H\pi_{i}\not\leq H$ for some $i\in\ir$. Then $H$ has a nontrivial strip $X\cong T$ such that
$|\supp(X)|\geq 2$.

We claim that $\supp(D_i)\nsubseteq\supp(X)$ for each $i\in\il$. 
To prove the claim,  assume that $Q_1,Q_2\in\supp(X)$
and consider the projection $\pi_1\colon S\rightarrow S_1$. As $S=S_\alpha H$, 
$$S_1=S\overline{\pi}_1=(S_\alpha \overline{\pi}_1)(H\overline{\pi}_1)=D_1(X\overline{\pi}_1).$$
However, we obtain, analyzing the orders, that
$$|S_1|=|D_1(X\overline{\pi}_1)|=
\frac{|D_1||X\overline{\pi}_1|}{|D_1\cap X\overline{\pi}_1|}=
\frac{|Q||T|}{|D_1\cap X\overline{\pi}_1|}<|Q|^2,$$
which is a contradiction. Thus,
$\supp(D_i)\nsubseteq\supp(X)$ for all $i\in\il$.
Again, renumbering if necessary, we may assume that $Q_2,Q_3\in\supp(X)$.
Considering the projection $\overline{\pi}\colon S\rightarrow S_1\times S_2$,
we have
\begin{equation}\label{s1s2}
S_1\times S_2=S\overline{\pi}=(S_\alpha \overline{\pi})(H\overline{\pi})=(D_1\times D_2)(H\overline{\pi}),
\end{equation}
where $H\overline{\pi}$ is contained in a subgroup $\overline{H}$ of $S_1\times S_2$ that is isomorphic to $T^3$. 
Set $P=(D_1\times D_2)\cap\overline H$. Then
$|Q|^4=|Q|^2|T|^3/|P|$,
thus
\begin{equation}\label{orderint}
 |P|=\frac{|T|^3}{|Q|^2}.
\end{equation}
Suppose that $Q\cong A_6$ and $T\cong A_5$. Then, by \eqref{orderint},
$|P|=5/3$, which is impossible.
If $Q\cong M_{12}$ and $T\cong M_{11}$, then the claim follows from
Lemma \ref{lemmam12}.

Assume that $Q\cong \po_8^+(q)$ and $T\cong \Omega_7(q)$. Then by (\ref{orderint})
$$|P|=q^3\cdot\frac{(q^6-1)\gcd(4,q^4-1)^2}{(q^2+1)\gcd(2,q-1)^3}.$$
We prove that there exists an odd prime that divides $q^2+1$ but does not divide $q^3\cdot (q^6-1)$.
If $q$ is even, then $q^2+1$ is odd, and so there exists an odd prime $p$ that divides $q^2+1$.
On the other hand, if $q$ is odd, then $q^2+1\equiv 2\mbox{ (mod $4$)}$.
Thus $q^2+1$ is even but it is not a 2-power, so there exists an odd prime $p$ that divides $q^2+1$.
Therefore, in both  cases, there exists an odd prime $p$ that divides $q^2+1$. 
We claim that $p$ does not divide $q^3\cdot (q^6-1)$. Since $p$ is odd,  $p$ does
not divide $q^2-1$. As $q^6-1=(q^2-1)(q^4+q^2+1)$, and 
$p$ does not divide $q^2-1$ but divides $q^2+1$, we find that $p$ does not divide $q^6-1$.
Hence $p$ is a prime we are looking for. This means that also 
$|P|$ is not an integer, which is  impossible. 

Therefore, $H=H\pi_1\times\cdots\times H\pi_r$ as asserted.

Thus  $1<H\pi_i<Q_i$ for all $i\in\ir$, $k=r$
and we may assume after possibly reordering the $T_i$ that
$T_i<Q_i$ for all $i$. The proof of Theorem~\ref{theorem3} is now complete.
\end{proof}

Now it only remains to prove Corollaries~\ref{cor1} and~\ref{cor3}.

\begin{proof}[The proof of Corollary~\ref{cor1}]
  In part (0) of Theorem~\ref{theorem3}, $H$ is not regular.
  In part~(2), $H=T_1\times T_2\leq Q^2$ such that $T_1T_2=Q$. The stabilizer
  of $H$ is isomorphic to $T_1\cap T_2$ which, as can be verified
  in each of the lines of Table~\ref{factiso}, is nontrivial. Hence
  $H$ is nonregular. Since, in case~(4), the group $G$ is a blow-up of a
  group that appears in case~(2), $H$ is nonregular in case~(4) also.
  Hence $G$ and $H$ must be as either in
  Theorem~\ref{theorem3}(1) or in Theorem~\ref{theorem3}(3) and the
  O'Nan--Scott type of $G$ is either {\sc As} or {\sc Pa}.
  \end{proof}

\begin{proof}[The proof of Corollary~\ref{cor3}]
  Let $H$, $G$ and $\Gamma$ be as in Corollary~\ref{cor3}. Then $H$ is
  a regular nonabelian and characteristically simple subgroup of the
  quasiprimitive group $G$ acting on the
  vertex set of $\Gamma$. Noticing that if $G$ is as in
  Theorem~\ref{theorem3}(1), then $G$ is either an alternating or a
  symmetric group, and so $\Gamma$, in this case, would be a
  complete graph, the result follows from Corollary~\ref{cor1}.
  \end{proof}

\bibliographystyle{alpha}

\newcommand{\etalchar}[1]{$^{#1}$}

\end{document}